\documentclass{amsart}

\usepackage{amssymb}
\usepackage{amsmath}
\usepackage{amsthm}
\usepackage{amsbsy}
\usepackage{bm}
\usepackage{hyperref}

\theoremstyle{plain}
\newtheorem{theorem}{Theorem}[section]
\newtheorem{lemma}[theorem]{Lemma}

\theoremstyle{definition}

\theoremstyle{remark}
\newtheorem{remark}[theorem]{Remark}

\newcommand{\Z}{\mathbb{Z}}
\newcommand{\R}{\mathbb{R}}

\renewcommand{\S}{\mathcal{S}}

\newcommand{\co}{\colon\thinspace}
\newcommand{\del}{\partial}

\newcommand{\hpi}{\widehat\pi}

\usepackage{tikz-cd}
\usepackage{xparse}
\usepackage[most]{tcolorbox}
\usepackage{amsthm}
\usepackage{todonotes}

\usepackage{mathtools}
\usepackage{adjustbox}
\usepackage{ccaption}

\begin{document}

\title[Goldman bracket characterizes homeomorphisms]{The Goldman bracket characterizes homeomorphisms between non-compact surfaces}

\author{Sumanta Das}
\email{sumantadas@iisc.ac.in}

\author{Siddhartha Gadgil}
\email{gadgil@iisc.ac.in}

\author{Ajay Kumar Nair}
\email{ajaynair@iisc.ac.in}

\address{	Department of Mathematics\\
  Indian Institute of Science\\
  Bangalore 560012, India}

\begin{abstract} We show that a homotopy equivalence between two non-compact orientable surfaces is homotopic to a homeomorphism if and only if it preserves the Goldman bracket, provided our surfaces are neither the plane nor the punctured plane. \end{abstract}
\maketitle

\section{Introduction}
All manifolds are assumed to be second countable and Hausdorff. A surface is a two-dimensional manifold. Throughout this note, all surfaces will be assumed to be connected and orientable. We say a surface is of \emph{finite-type} if its fundamental group is finitely generated; otherwise, we say it is of \emph{infinite-type}.

A fundamental question in topology is whether homotopy equivalent closed $n$-manifolds are necessarily homeomorphic, and whether every homotopy equivalence is homotopic to a homeomorphism. More generally, we may ask what additional structures or conditions characterize homotopy equivalences that are homotopic to homeomorphisms. For $n\geq 3$, the answer to these questions is negative in general -- for example, the Lens spaces $L(7, 1)$ and $L(7, 2)$ are homotopy equivalent but not homeomorphic.

In the case of closed surfaces, the classical \emph{Dehn-Nielsen-Baer theorem}~\cite[Appendix]{MR0881797} says that every homotopy equivalence is homotopic to a homeomorphism. However, the corresponding result does not hold for compact two-manifolds with boundary. For example, the  torus with one hole and the sphere with three holes are homotopy equivalent but not homeomorphic.

Similarly, homotopy equivalence does not imply homeomorphism for non-compact surfaces without boundary. For example, the once-punctured torus and the thrice-punctured sphere are homotopy equivalent but not homeomorphic. Indeed up to homotopy equivalence, there is precisely one (connected) surface  of infinite type, but up to homeomorphism, there are $2^{\aleph_0}$ many infinite-type surfaces; see \cite[Proposition 3.1.11]{arxiv1}.

In this note, we show that there is a simple and natural characterization of when a homotopy equivalence $f\colon \Sigma'\to\Sigma$ between non-compact surfaces without boundary is homotopic to a homeomorphism in terms of the \emph{Goldman bracket}, which is a Lie Algebra structure associated to a surface. More precisely, if $\hpi(\Sigma)$ denotes the set of free homotopy classes of closed curves on a non-compact surface $\Sigma$, then the Goldman bracket (whose definition we recall in Section~\ref{gold-intro}) is a bilinear map
$$[\cdot,\cdot]\co \mathbb Z[\hpi(\Sigma)]\times\mathbb Z[\hpi(\Sigma)]\to\mathbb Z[\hpi(\Sigma)],$$
which is skew-symmetric and satisfies the Jacobi identity. Our main result is the following.

\begin{theorem}\label{main}
  A homotopy equivalence $f\co \Sigma'\to \Sigma$ between two non-compact oriented surfaces without boundary is homotopic to an orientation-preserving homeomorphism if and only if it commutes with the Goldman bracket, i.e., for all $x',y'\in\mathbb Z[\hpi(\Sigma')]$, we have,
  \begin{equation}\label{preserve}
    \left[f_*(x'),f_*(y')\right]=f_*\left([x',y']\right),
  \end{equation}
  where $f_*\co \mathbb Z[\hpi(\Sigma')]\to\mathbb Z[\hpi(\Sigma)]$ is the function induced by $f$, provided $\Sigma$ is not homeomorphic to the plane or the cylinder $S^1 \times \R$.
\end{theorem}

The analogous result for compact, connected, oriented two-dimensional manifolds with boundary was proved in~\cite{MR2861998}.

Our methods are based on the relation between the Goldman bracket and the \emph{geometric intersection number} of curves on a surface. The same methods prove a related characterization in terms of intersection numbers for homotopy equivalences.

Recall that, for a $2$-manifold $M$ with or without boundary, the geometric intersection number $I_M$ is defined as follows. Let $x, y\in \widehat\pi(M)$. Suppose $\alpha: \S^1\to M$ and $\beta: \S^1\to M$ are smooth immersions representing $x$ and $y$, respectively. We say that $\alpha$ and $\beta$ are in \emph{general position} if they intersect transversally so that if $p \in \alpha(\S^1)\cap \beta(\S^1)$, then $\alpha^{-1}(p)$ and $\beta^{-1}(p)$ are both singletons. We sometimes simply say $\alpha$ and $\beta$ intersect in double points.

We define the geometric intersection number as
$$I_M(x,y)\coloneqq \min\left\{|\alpha\cap \beta|: \text{$\alpha\in x$, $\beta\in y$, $\alpha$ and $\beta$ in general position}\right\}.$$

\begin{theorem}\label{intersection}
  Let $f\co \Sigma'\to \Sigma$ be a homotopy equivalence between two non-compact surfaces without boundary, where $\Sigma$ is not homeomorphic to the plane or a cylinder. Then the following are equivalent:
  \begin{enumerate}
    \item\label{it:homeo} $f$ is homotopic to a homeomorphism.
    \item\label{it:inter}  $I_\Sigma\left(f_*(x'), f_*(y')\right)=I_{\Sigma'}(x', y')$ for all $x', y'\in \widehat\pi(\Sigma')$.
    \item\label{it:zero-inter} For all $x', y'\in \widehat\pi(\Sigma')$, $I_\Sigma\left(f_*(x'), f_*(y')\right)=0 \iff I_{\Sigma'}(x', y')=0$.
  \end{enumerate}
\end{theorem}
\subsection{Outline of the proof of Theorem~\ref{main}}\label{outline}

As orientation-preserving homeomorphisms preserve the Goldman bracket by definition, it is easy to see that if $f\co \Sigma' \to \Sigma$ is homotopic to an orientation-preserving homeomorphism, then it commutes with the Goldman bracket.

Conversely, suppose $f\co \Sigma' \to \Sigma$ commutes with the Goldman bracket, we construct a \emph{proper} map $g\co  \Sigma'\to \Sigma$ such that $f$ is homotopic to $g$. By~\cite[Theorem 1]{arxiv1}, $g$ is in turn (properly) homotopic to a homeomorphism.

To construct $g$, we pick an exhaustion $K_1 \subset K_2\subset \dots \subset K_n\subset \dots$ of $\Sigma$ (see Section~\ref{exhaustions}) by compact subsurfaces (in this outline, we suppress a few technical conditions on exhaustions and base-points). We construct (as we sketch below) an exhaustion $K'_1 \subset K'_2\subset \dots \subset K'_n\subset \dots$ of $\Sigma'$ by compact subsurfaces together with maps $g_i \co  K'_i \to \Sigma$ homotopic to the restrictions $f\vert_{K_i'}$ such that for all $i$, $g_{i+1}$ is an extension of $g_i$, i.e., $g_{i+1}\vert_{K_i'} = g_i$. Furthermore, we ensure that $g_n(K_n' \setminus K_i')\subset \Sigma\setminus K_i$ whenever $i\leq n$.

To construct $K_1'$ and the map $g_1\co K_1' \to \Sigma$, we consider a system of simple closed curves $\alpha_1$, $\alpha_2$, \dots, $\alpha_{k_1}$ that \emph{fill} $K_1$ (see Section~\ref{filling}). Thus, if $\gamma$ is a closed curve that can be homotoped to be disjoint from all $\alpha_i$, then $\gamma$ is homotopic to a closed curve in $\Sigma\setminus K_1$. As $f$ is a homotopy equivalence, there exist closed curves $\alpha'_1$, $\alpha'_2$, \dots, $\alpha'_{k_1}$ such that $f_*(\alpha'_i)$ is homotopic to $\alpha_i$ for all $i$. We pick a compact subsurface $K_1'$ of $\Sigma'$ that contains all the curves $\alpha_i'$.

The key observation is that for any curve $\gamma'\subset \Sigma'\setminus K_1'$, $f_*(\gamma')$ is homotopic to a curve in $\Sigma\setminus K_1$. This follows from the properties of the Goldman bracket. Namely, the Goldman bracket of $\gamma'$ with each $\alpha_i'$ is zero, and hence the Goldman bracket of $f_*(\gamma')$ with each $\alpha_i$ is zero. As the $\alpha_i$ are simple closed curves, by a theorem of Goldman, it follows that, for each $i$, $f_*(\gamma')$ is homotopic to a curve disjoint from $\alpha_i$. As the $\alpha_i$ fill $K_1$, it follows that $f_*(\gamma')$ is homotopic to a curve in $\Sigma\setminus K_1$.

As the boundary components of $K_1'$ are homotopic to curves in $\Sigma'\setminus K_1'$, this lets us define a map $g_1\co K_1' \to \Sigma$ that is homotopic to $f\vert_{K_1'}$ and so that $g_1(\del K_1') \subset \Sigma\setminus K_1$. To proceed inductively, we prove a refinement of the above key observation. Namely, for each component $V'$ of $\overline{\Sigma'\setminus K_1'}$, there is a component $V$ of $\overline{\Sigma\setminus K_1}$ such that for any closed curve $\gamma'\subset V'$, $f_*(\gamma')$ is homotopic to a curve in $V$. This lets us make a construction similar to the above for each component $V'$ of $\overline{\Sigma'\setminus K_1'}$ to obtain a subsurface in $V$ and an extension of the map to this subsurface. Taking the union of these subsurfaces, we get $K_2'$ and an extension $g_2$ of $g_1$ to $K_2'$ so that $g_2(\del K'_2)\subset \Sigma\setminus K_2$. We proceed inductively to obtain $K'_n$ and maps $g_n: K'_n \to \Sigma$ with $g_n(\del K'_n)\subset \Sigma\setminus K_n$ for all $n$.

Finally, the limit of the maps $g_i$ gives us a proper map $g\co  \Sigma'\to \Sigma$ homotopic to $f$ as each $g_i \co  K'_i \to \Sigma$ homotopic to the restriction $f\vert_{K_i'}$.

\section{Background}

A \emph{surface} $\Sigma$ is a connected, orientable two-dimensional manifold with or without boundary.

\subsection{Subsurfaces}

A (compact) \emph{subsurface} $F$ of a surface $\Sigma$ is an embedded submanifold (in general with boundary) of codimension zero. Thus, $F\subset\Sigma$ is a subset homeomorphic to a compact surface with boundary so that $\del F\subset \Sigma$ is a collection of disjoint simple closed curves. Further, if $F$ is connected, closure of a component of $\Sigma\setminus \del F$  is $F$.

We assume that all subsurfaces are connected.

\subsection{Exhaustions}\label{exhaustions}

An \emph{exhaustion} of a non-compact surface $\Sigma$ is a collection of subsurfaces $K_i$, $i=1, 2,\dots$ such that
\begin{itemize}
  \item $K_i\subset \mathrm{int}(K_{i+1})$ for all $i$,
  \item $\bigcup_i K_i = \Sigma$.
\end{itemize}

\subsection{Intersections and geodesics}\label{S:int-geods}

A very useful classical result on surfaces relates pairwise and mutual intersections. This can be conveniently formulated in terms of hyperbolic metrics on surfaces (one can define this purely topologically, but the definition is a bit more complicated).

Recall that except sphere and torus, every compact surface $\Sigma$ without boundary admits a complete hyperbolic metric. Similarly, every compact surface $\Sigma$ with boundary admits a hyperbolic metric with geodesic boundary if it has negative Euler characteristic. Assume that such a metric has been fixed. Then the following holds.

\begin{lemma}\label{int-geods}
  Let $\Sigma$ be a surface with a hyperbolic metric
  \begin{enumerate}
    \item Every non-trivial homotopy class of curves in $\Sigma$ contains a unique geodesic representative.
    \item Every homotopically non-trivial simple closed curve $\alpha$ in $\Sigma$ is ambient isotopic to a simple geodesic.
    \item If two homotopy classes of curves in $\Sigma$ have representatives that are disjoint, then their geodesic representatives are disjoint.
  \end{enumerate}
\end{lemma}

It follows that for a finite collection of homotopy classes of curves $x_1, x_2, \dots, x_n$ in $\Sigma$, if every pair $x_i$ and $x_j$ have disjoint representatives, then there exists a family of mutually disjoint representatives of $x_1, x_2, \dots, x_n$.

In particular, if $F\subset \Sigma$ is a subsurface, then the geodesics homotopic to the components of $\del F$ are pairwise disjoint, and the closure of a component of $\Sigma\setminus \del F$ is a hyperbolic surface with geodesic boundary which is isotopic to $F$. We can thus assume that subsurfaces have geodesic boundary.

\subsection{Filling curves}\label{filling}

Let $F$ be a compact surface, possibly with boundary. A homotopically non-trivial closed curve $\alpha\subset F$ is said to be \emph{peripheral} if $\alpha$ is (freely) homotopic to a curve with the image in $\del F$.

Fix a hyperbolic metric on $F$ with geodesic boundary. We say that a collection of geodesics $\alpha_1$, $\alpha_2$, \dots, $\alpha_n$ in $F$ \emph{fills} $F$ if every component of $F\setminus\bigcup_i\left(\alpha_i\right)$ is either an open disc or a punctured disc with boundary contained in $\del F$. By the results recalled in Lemma~\ref{int-geods}, if a homotopically non-trivial curve $\gamma$ is homotopic to curves disjoint from each $\alpha_i$, $1\leq i\leq n$, then $\gamma$ is peripheral.

Similarly, if $F\subset \Sigma$ is a compact subsurface with geodesic boundary and the geodesics $\alpha_1$, $\alpha_2$, \dots, $\alpha_n$ in $F$ fill $F$, then every homotopically non-trivial curve $\gamma$ in $\Sigma$ that is disjoint from each $\alpha_i$ is homotopic to a curve in $\Sigma\setminus \textrm{int}(F)$, and hence to a curve in $\Sigma\setminus F$.

\subsection{Splitting surfaces}\label{splitting}

Let $\Sigma$ be a surface, and $\gamma$ be a simple closed curve on $\Sigma$ so that $\Sigma\setminus \gamma$ has two components. Observe that a curve $\alpha$ in a component $V$ of $\Sigma\setminus \gamma$ is  peripheral in $V$ if it is homotopic in $\Sigma$ to a curve in a regular neighborhood of $\gamma$.

Using the results recalled in Lemma~\ref{int-geods}, we deduce the following.

\begin{lemma}\label{split}
  Let $\alpha\subset \Sigma$ be a curve that is disjoint from $\gamma$ and is not peripheral. If $\beta\subset \Sigma\setminus\gamma$ is a closed curve that is homotopic to $\alpha$ in $\Sigma$, then $\alpha$ and $\beta$ are contained in the same component $V$ of $\Sigma\setminus \gamma$ and are homotopic in $V$.
\end{lemma}

\subsection{Goldman bracket}\label{gold-intro}
Let $\Sigma$ be a non-compact surface, and let $\widehat\pi(\Sigma)= [\mathbb S^1,\Sigma]$, i.e., $\widehat\pi(\Sigma)$ denotes the set of free homotopy classes of curves in $\Sigma$. Note that if $p\in \Sigma$, there is a bijection between the set of all conjugacy classes of $\pi_1(\Sigma,p)$ and $\widehat\pi(\Sigma)$. For a closed curve $\alpha$, let $\widehat{\alpha}$ denote the free homotopy class of $\alpha$.

Let $x,y\in \widehat\pi(\Sigma)$ and $\alpha, \beta$ be smoothly immersed, oriented representatives of $x$ and $y$, respectively, such that  $\alpha$ and $\beta$ intersect transversally in (finitely many) double points. Given any point $p\in \alpha\cap \beta$, we view $\alpha$ and $\beta$ as curves in $\Sigma$ based at $p$ and define $\alpha *_p \beta$ as the product of based curves as in the definition of the fundamental group.

We associate to each intersection point $p$ a \emph{sign} $\varepsilon_p =\pm 1$ as in the definition of algebraic intersection number. Namely, parametrizing $\alpha$ and $\beta$ so that $\alpha(0)=\beta(0)=p$, if the ordered pair $(\alpha'(0), \beta'(0))$ represents the orientation of $\Sigma$ at $p$, then  $\varepsilon_p = 1$, otherwise  $\varepsilon_p = -1$.

The Goldman bracket~\cite{MR0846929} $[\cdot,\cdot]\colon\mathbb Z\big[\widehat\pi(\Sigma)\big]\times\mathbb Z\big[\widehat\pi(\Sigma)\big]\to \mathbb Z\big[\widehat\pi(\Sigma)\big]$ is defined as
\begin{equation}
  [x,y]=\sum_{p\in \alpha\cap \beta}\varepsilon_p\cdot  \widehat{\alpha*_p \beta}.  \label{Goldman}
\end{equation}

Recall that for a non-empty set $S$, $\mathbb Z[S]$ is the free module generated by $S$ over the ring of integers $\mathbb Z$. Also, note that the sum in equation~\ref{Goldman} is a finite sum.

\begin{theorem}[Goldman~{\cite{MR0846929}}]\label{goldman}
  The Goldman bracket satisfies the following properties:

  \begin{enumerate}
    \item $[\cdot,\cdot]\colon\mathbb Z\big[\widehat\pi(\Sigma)\big]\times\mathbb Z\big[\widehat\pi(\Sigma)\big]\to \mathbb Z\big[\widehat\pi(\Sigma)\big]$ is a well-defined bilinear map; see \cite[Theorem 5.2.]{MR0846929}.
    \item The bracket is skew-symmetric and satisfies the Jacobi identity, i.e.,
          \begin{enumerate}
            \item $[x,y]=-[y,x]$ for all $x,y\in \widehat\pi(\Sigma)$, and
            \item $\big[x,[y,z]\big]+ \big[y, [z,x]\big]+ \big[z, [x,y]\big]=0$ for all $x,y,z\in \widehat\pi(\Sigma)$; see \cite[Theorem 5.3.]{MR0846929}.
          \end{enumerate}
          \item\label{it:gold-scc} If $\alpha$ is a simple closed curve and $x= \widehat{\alpha}$ is the homotopy class of $\alpha$, then, for $y\in \widehat \pi(\Sigma)$, $[x, y] = 0$ if and
          only if $y=\widehat{\beta}$ for a closed curve $\beta$ such that $\beta\cap \alpha=\varnothing$; see \cite[Theorem 5.17. (i)]{MR0846929}.
  \end{enumerate} \label{GoldmanBracket}
\end{theorem}

We will frequently abuse notation and speak of the Goldman bracket of two curves $\alpha$ and $\beta$ in a surface $\Sigma$ to mean the Goldman bracket of their free homotopy classes  $\widehat{\alpha}$ and $\widehat{\beta}$ in $\widehat\pi(\Sigma)$.

\section{Proof of Theorem \ref{main}}

Suppose $f$ is homotopic to an orientation-preserving homeomorphism $g\co  \Sigma'\to \Sigma$. Without loss of generality,  we may assume $g$ is an orientation-preserving diffeomorphism since in dimension two, any homeomorphism is isotopic, thus properly homotopic to a diffeomorphism. By the definition of the Goldman bracket, it follows that $g$ commutes with the Goldman bracket. As $f$ and $g$ are homotopic, it follows that $f$ also commutes with the Goldman bracket. This proves one implication of Theorem~\ref{main}.

Conversely, suppose that $f\co  \Sigma'\to \Sigma$ is a homotopy equivalence between non-compact surfaces such that $f$ commutes with the Goldman bracket. Pick base-points $p'\in \Sigma'$ and $p\in \Sigma$ so that $f(p')=p$. As sketched in Section~\ref{outline}, we will construct a proper map $g$ that is homotopy equivalent to $f$. This is an inductive construction, where we construct surfaces $K_m'$  and maps $g_m\co K_m'\to\Sigma$. To keep track of base-points, we will also construct inductively a family of trees $T_m'\subset K_m'$ starting with $T_0'={p'}$ so that $T_{m}'\cap K_{m-1}'= T_{m-1}'$ for $m > 1$ and the terminal vertices of $T_m'$ lie on $\del K_m'$ with one terminal vertex in each boundary component. We will also modify $f$ by a homotopy in a neighborhood of the trees $T_m'$ (with $f$ fixed throughout the homotopy on $p'$, and on the set $K_{m-1}$ for $m > 1$).

Pick an exhaustion $K_1 \subset K_2\subset \dots \subset K_n\subset \dots$ of $\Sigma$ by compact subsurfaces with boundary such that $p\in K_1$ and $K_1$ (and hence every $K_m$) is not a disc or an annulus -- this is possible as $\Sigma$ is not a plane or a cylinder.
Also choose an exhaustion $L_1'\subset L_2'\subset \dots \subset L_n'\subset \dots$ of $\Sigma'$ by compact bordered subsurfaces such that $p'\in L_1'$

We first construct a compact subsurface $K_1'\subset \Sigma'$, a map $g_1\co  K_1'\to \Sigma$, and a tree $T'_1\subset K_1'$, which satisfy various properties. These properties will allow us to continue the construction inductively and to ensure that the resulting map is proper.

\subsection{Constructing the subsurface}\label{ExtendedFilling}

The construction of the subsurface is almost identical in the base case and the general case, so we describe the general case directly. Fix $m\geq 1$. If $m > 1$, assume that a subsurface $K_{m-1}'$ has been constructed. Let $Q_m' \coloneqq L_1'$ if $m=1$ and $Q_m' \coloneqq L_m'\cup K_{m-1}'$ if $m > 1$.

Consider a system of simple closed curves $\alpha_1$, $\alpha_2$, \dots, $\alpha_{l}$ that fill $K_m$ (as defined in Section~\ref{filling}). Thus, if $\gamma$ is a closed curve that can be homotoped to be disjoint from each of the curves $\alpha_i$, then $\gamma$ is homotopic to a closed curve in $\Sigma \setminus K_m$. We enlarge this collection to ensure that a peripheral curve must intersect a curve in the system. Namely, if $C\subset \partial K_m$ is a boundary component of $K_m$ which is not the boundary of a component in $\overline{\Sigma\setminus K_m}$ homeomorphic to the half-open cylinder $[0, \infty)\times \R$ or the closed unit disc, then there exists a simple closed curve $\alpha^{(C)}$  so that $C$ and $\alpha^{(C)}$ cannot be homotoped to be disjoint. Add such a curve for each boundary component of $K_m$ that is not the boundary of a cylinder or disc in $\Sigma$. Let $\alpha_1$, $\alpha_2$, \dots, $\alpha_{l}$, $\alpha_{l+1}$, \dots, $\alpha_{k}$ be the resulting collection of curves; see Figure~\ref{drawing4}. 

\begin{figure}$$\adjustbox{trim={0.0\width} {0.0\height} {0.0\width} {0.0\height},clip}{\def\svgwidth{\linewidth}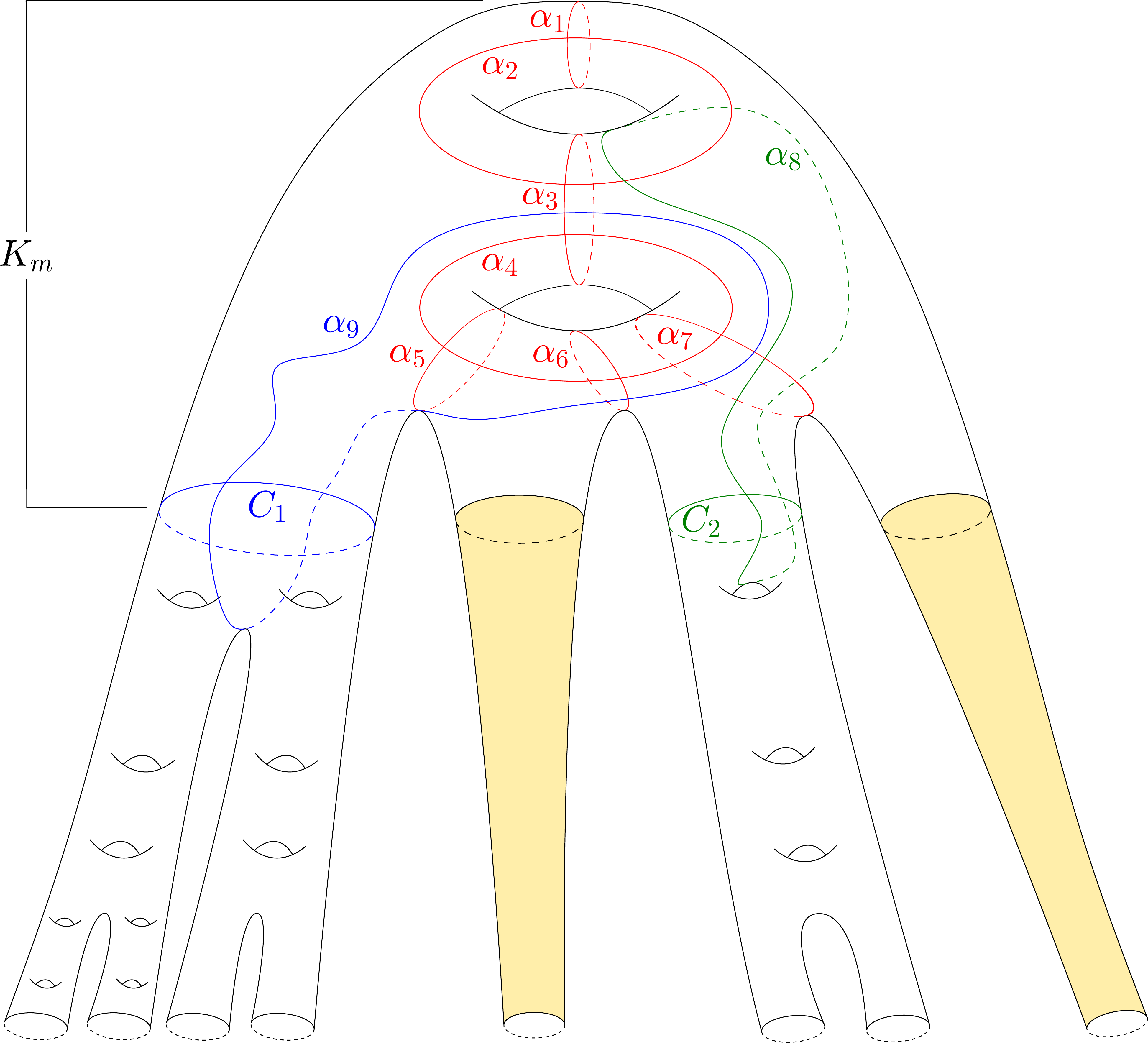}$$\caption{$\left\{\alpha_1,...,\alpha_7\right\}$ is a filling system of $K_m$ and $\left\{\alpha_1,...,\alpha_9\right\}$ is an extended filling system for $K_m$ because $I_\Sigma\left(C_1,\alpha_9\right)\neq 0\neq I_\Sigma\left(C_2,\alpha_8\right)$. }\label{drawing4}
\end{figure}

As $f$ is a homotopy equivalence, there exist closed curves $\alpha'_i\subset \Sigma'$, $1\leq i\leq k$, such that $f_*(\alpha'_i)$ is homotopic to $\alpha_i$. Further, we can assume that the curves $\alpha'_i$ are pairwise transversal and intersect only in double points.

\begin{lemma}\label{subsurface}
  There exists a compact subsurface $K_m'\subset \Sigma'$ that satisfies the following:
  \begin{enumerate}
    \item For all $i$, $1\leq i\leq k$, $\alpha_i'\subset K_m'$.
    \item $p'\in K_m'$.
          \item\label{it:big} $Q_m'\subset K_m'$.
          \item\label{it:noncpt} The closure of every component of $\Sigma'\setminus K_m'$ is non-compact.
          \item\label{it:connected} Every component of $\overline{\Sigma'\setminus K_m'}$ intersects $K_m'$ in a single component of $\del K_m'$.
  \end{enumerate}
\end{lemma}

\begin{proof}
  To ensure the first three conditions, we simply take a compact connected subsurface $K_m'$ that contains $\left(\bigcup_i \alpha'_i\right)\cup Q'_{m-1}\cup \{p'\}$.

  Next, as $K_m'$ is compact, the closure $\overline{\Sigma'\setminus K_m'}$ of its complement has finitely many components $V'_1$, $V'_2$, \dots, $V'_n$. Suppose the compact complementary components are $V'_{i_1}$, $V'_{i_2}$,\dots $V'_{i_l}$, then we replace $K_m'$ by $K_m' \cup \left(\bigcup\limits_{j=1}^n V'_{i_j}\right)$. This is also compact and satisfies condition~(\ref{it:noncpt}).

  Finally, let $V'_1$, $V'_2$, \dots, $V'_l$ be the (non-compact) components of $\overline{\Sigma'\setminus K_m'}$. Suppose $V'_i$ intersects $K_m'$ in $n_i > 1$ components of $\del K_m'$, say $C_1'$, $C_2'$, \dots $C_{n_i}'$. For $j=2, 3, \dots n_i$, pick disjoint simple arcs $\beta_j'$ in $V'_j$ connecting $C_1'$ to $C_j'$ .  Replace $K_m'$ by $K_m' \cup \mathcal{N}\left(\bigcup\limits_{j=2}^{n_i} \beta_j'\right)$, where $\mathcal{N}$ denotes a regular neighbourhood in $V'_i$. Making a similar construction for each $V'_i$ that intersects $K_m'$ in more than one boundary component, we obtain a compact subsurface $K_m'$ that satisfies the condition~(\ref{it:connected}) (and continues to satisfy the other conditions).
\end{proof}

\begin{remark} \label{additional}
   The proof of Lemma~\ref{subsurface} shows that any non-compact surface $\Sigma$ without boundary has an exhaustion $\{K_i\}$ by compact subsurfaces such that for each $i=1,2,...$, the following hold: 
   \begin{itemize}
       \item the closure of every component of $\Sigma\setminus K_i$ is non-compact, and 
       \item every component of $\overline{\Sigma\setminus K_i}$ intersects $K_i$ is a single component of $\partial K_i$. 
   \end{itemize}
\end{remark}
\subsection{Pushing loop images outside subsurfaces}
Choose and fix a compact subsurface $K_m'$ satisfying the conditions of Lemma~\ref{subsurface}. Using the assumption that $f$ commutes with the Goldman bracket gives the following key lemma.

\begin{lemma}\label{loop_outside}
  Let $\gamma'$ be a closed curve in $\Sigma'$ that is homotopic to a closed curve in $\Sigma'\setminus K_m'$. Then $f_*(\gamma')$ is homotopic to a closed curve in $\Sigma\setminus K_m$.
\end{lemma}
\begin{proof}
  Since $\gamma'$ is homotopic to a closed curve in $\Sigma'\setminus K_m'$, $\gamma'$ is homotopic to a curve that is disjoint from $\alpha_i'$ for all $i$, $1\leq i\leq k$. It follows that the Goldman bracket $[\gamma', \alpha_i']$ is zero for all $i$, $1\leq i\leq k$.

  As $f$ commutes with the Goldman bracket, $[f_*(\gamma'), \alpha_i]=0$ for all $i$. For each $i$, as $\alpha_i$ is a simple closed curve, it follows that $f_*(\gamma')$
  is homotopic to a curve disjoint from $\alpha_i$ by statement (\ref{it:gold-scc}) of Theorem~\ref{GoldmanBracket}. As a subset of the curves $\alpha_i$ fill $K_m$, it follows that $f_*(\gamma')$ is homotopic to a curve in $\Sigma\setminus K_m$.
\end{proof}

\subsection{Construction of the map in the base case}\label{S:base}
We next construct a map $g_1\co K_1'\to \Sigma$  that is homotopic to $f\vert_{K_1'}$ and satisfies $g_1(p')=p$ so that $g_1(\partial K_1')\cap K_1=\varnothing$. As $K_1'$ is a compact surface with a non-empty boundary, there exists a graph  $\Gamma'\subset K_1'$ such that $K_1'$ deformation retracts onto $\Gamma'$. We can choose $\Gamma'$ so that $p'\in \Gamma'$ and $\Gamma'\cap \partial K_1'=\varnothing$. A regular neighbourhood $\mathcal{N}(\Gamma')$ of $\Gamma'$ has a deformation retraction $\rho\co \mathcal{N}(\Gamma')\to \Gamma'$. We get a map $g_1\co \mathcal{N}(\Gamma')\to \Sigma$ by setting $g_1\vert_{\mathcal{N}(\Gamma')}=f\circ \rho$.

Further, $\overline{K_1'\setminus \mathcal{N}(\Gamma')}$ is a disjoint union of annuli $C_i'\times [0, 1]$, $1\leq i\leq n$, so that the boundary components of $K_1'$ are the curves $C_i' \times \{1\}$, $1\leq i\leq n$ and the boundary components of $\mathcal{N}(\Gamma')$ are $C_i' \times \{0\}$, $1\leq i\leq n$. By Lemma~\ref{loop_outside}, for each $i$, $1\leq i\leq n$, there is a curve $\gamma_i\subset \Sigma\setminus K_1$ so that $f(C_i')$ is homotopic to $\gamma_i$. We define $g_1$ on each annulus $C_i'\times [0, 1]$ as follows. We set $g_1(C_i'\times \{1\})=\gamma_i$. Note that $C_i'\times \{0\}$ is  identified with a boundary component of $\mathcal{N}(\Gamma')$, and so $g_1$ has been defined on $C_i'\times \{0\}$. Finally, by construction  $g_1\vert_{C_i'\times \{0\}}$ and $g_1\vert_{C_i'\times \{1\}}$ are both homotopic in $\Sigma$ to $f(C_i')$ and hence homotopic to each other -- so we can extend $g_1$ to $C'_i\times [0,1]$ using a homotopy in $\Sigma$.

We next construct the tree $T_1'$ and modify the map $f$. Let the boundary components of $K_1'$ be $\delta_i'= C_i'\times \{1\}$, $1\leq i\leq n$. Pick a point $q_i'\in \delta_i'$ for each $i$, $1\leq i\leq n$. Pick a family of arcs $\theta_i'$ with interiors disjoint so that $\theta_i'$ is an arc from $p'$ to $q_i'$ in $K_1'$. Let $T_1'$ be the union of these arcs. As $T_1'$ is contractible and $f(p') = g_1(p')$, we can homotope $f$ fixing $p'$ so that $f\vert T_1'=g_1\vert T_1'$.

Note that for $1\leq i\leq n$,  $f(q_i')=g_1(q_i')$. Further, for $1\leq i\leq n$, as $(f\vert_{K_1'})_*=(g_1)_*\co \pi_1(K_1', p') \to \pi_1(\Sigma, f(p'))$,  and $f\vert_{\theta_i'}=g_1\vert_{\theta_i'}$, we deduce by using the \emph{change of base-point isomorphisms} corresponding to $\theta_i'$ and the fact that $f\circ {\theta_i'}=g_1\circ {\theta_i'}$ that $(f\vert_{K_1'})_*=(g_1)_*\co \pi_1(K_1', q') \to \pi_1(\Sigma, f(q'))$.

\subsection{Inductive properties}\label{S:inductive}

We now consider the inductive step, where $m > 1$, and we have already defined a subsurface $K'_{m-1}\subset \Sigma'$, a map $g_{m-1}\co K'_{m-1}\to \Sigma$, and a tree $T'_{m-1}\subset K'_{m-1}$. Further, we assume that these satisfy the following properties:
\begin{enumerate}
  \item The subsurface $K_{m-1}'$ satisfies the conditions of Lemma~\ref{subsurface}.
  \item $g_{m-1}(\del K_{m-1}')\subset \Sigma\setminus K_{m-1}$.
  \item $f\vert_{T'_{m-1}}=g_{m-1}\vert_{T'_{m-1}}$.
  \item The terminal vertices of $T'_{m-1}$ are in $\del K_{m-1}'$, with one terminal vertex in each component.
  \item For a terminal vertex $q'\in T'_{m-1}$, $(f\vert_{K_{m-1}'})_*=(g_{m-1})_*:\pi_1(K'_{m-1}, q') \to \pi_1(\Sigma, q)$.
\end{enumerate}

Observe that we have shown in Section~\ref{S:base} that the above properties hold for $m=2$. We will construct $K'_m$, $g_m$, and $T'_m$ and show (in Lemma~\ref{construction}) that the above properties hold for these.

\subsection{Maps on fundamental groups of components}\label{S:pi_map}

Let $K'_m$ be the subsurface constructed using Lemma~\ref{subsurface}. Let $V'^{(1)}$, $V'^{(2)}$, \dots $V'^{(k)}$ be the components of $\overline{K'_m\setminus K'_{m-1}}$. By Lemma~\ref{subsurface}, each component $V'^{(j)}$ intersects $K'_{m-1}$ in a single boundary component $\delta'^{(j)}$.

We construct extensions $g_m^{(j)}\co K'_{m-1}\cup V'^{(j)}\to \overline{\Sigma\setminus K_{m-1}}$ of $g_{m-1}$ for each component $V'^{(j)}$, $1\leq j \leq k$ in Section~\ref{S:complementary}. We first show that we have homomorphisms between fundamental groups of components of $\overline{K'_m\setminus K'_{m-1}}$ and components of $\overline{\Sigma\setminus K_{m-1}}$.

Fix a component $V'=V'^{(j)}$ of $\overline{K'_m\setminus K'_{m-1}}$  and let $\delta'=\delta'^{(j)}$. Let $i': V'\to \Sigma'$ be the inclusion map.
By construction $g_{m-1}(\delta')\subset \Sigma\setminus K_{m-1}$. Let $V$ be the component of $\Sigma\setminus K_{m-1}$ containing $g_{m-1}(\delta')$ and $i:V\to \Sigma$ be the inclusion map. We construct an extension with image contained in $V$. Most of the additional work beyond Lemma~\ref{loop_outside} is to control images of curves up to \emph{based} homotopy, not just free homotopy.

Let $q'\in V'$ be the terminal vertex of $T'_{m-1}$ corresponding to $V'$, and let $q= f(q')$. We have $(f\vert_{K_{m-1}'})_*=(g_{m-1})_*\co \pi_1(K_{m-1}', q')\to \pi_1(\Sigma, q)$.
Observe that $\delta'$ can be regarded as a closed curve based at $q'$. As the image of $\delta'$ is contained in $V$,  $(f\vert_{K_{m-1}'})_*([\delta'])=(g_{m-1})_*([\delta']) \in i_*(\pi_1(V, q))$.
The key idea is to apply Lemma~\ref{loop_outside} to both the curves $\gamma'$ and $\gamma'*\delta'$ for a curve $\gamma'\subset V'$ based at $q'$. We first need a lemma about amalgamated free products.

Let $A$ and $B$ be groups with $C\subseteq A$ and $C\subseteq B$ as subgroups of both of them. Let $G\coloneqq A*_C B$ be  the \emph{amalgamated free product} of $A$ and $B$ along $C$(see~\cite{MR0207802}). The groups $A$ and $B$ are called the \emph{factors} of $G$. Following terminology from topology, we say that an element in $A$ or $B$ is \emph{peripheral} if it is conjugate to an element in $C$.

\begin{lemma}\label{amalgamated_free_product}
  Let $a\in A$ be an element such that $a$ is not conjugate to an element in $C$. Let $g\in G$.
  \begin{enumerate}
    \item\label{it:diff} Suppose $g$ is conjugate to a non-peripheral element $b\in B$. Then $ag$ is not conjugate to an element in a factor of $G$.
    \item\label{it:same} If $g$ is conjugate to a non-peripheral element $a'$ in $A$ and $ag$ is conjugate to an element in a factor, then $g\in A$.
  \end{enumerate}
\end{lemma}
\begin{proof}
  By \cite[Theorem 4.6]{MR0207802}, if $ag$ is conjugate to an element in a factor, then any cyclically reduced word representing an element conjugate to $ag$ is contained in a factor.

  First, suppose $g$ is conjugate to a non-peripheral element $b\in B$, say $g=hbh^{-1}$. We consider a reduced word $h=l_1\cdot l_2\cdot l_3\dots\cdot l_k$ representing $h$ (we will separate \emph{letters} in words in the free products by $\cdot$). There are a few different cases corresponding to in which factor the first and last letters of $h$ belong.

  We first consider the most non-trivial case, where $h=a_1\cdot b_1\cdot\ldots\cdot a_k\cdot b_k$, so $g=a_1\cdot b_1\cdot\ldots\cdot a_k\cdot b_k\cdot b \cdot b_k^{-1}\cdot a_k^{-1}\cdot\ldots\cdot b_1^{-1}\cdot a_1^{-1}$. We claim that a cyclically reduced word conjugate to $ag$ is $(a_1^{-1}aa_1)\cdot b_1\cdot\ldots\cdot a_k\cdot (b_k b b_k^{-1}) \cdot a_k^{-1}\cdot\ldots\cdot b_1^{-1}$. This clearly represents the element $a_1^{-1}a g a_1$, which is conjugate to $ag$. As $a$ and $b$ are not peripheral, $a_1^{-1}aa_1\notin C$ and $b_k b b_k^{-1}\notin C$, hence the word is indeed cyclically reduced. As this word is not contained in a factor, $ag$ is not conjugate to an element in a factor.

  The other three cases are analogous. In each case we obtain a cyclically reduced word that is not in a factor. For completeness, we list the cases and the cyclically reduced words we obtain:
  \begin{enumerate}
    \item $h=a_1\cdot b_1\cdot\ldots\cdot a_k$, $k\geq 1$, we obtain $(a_1^{-1}aa_1)\cdot b_1\cdot\ldots\cdot a_k\cdot b \cdot a_k^{-1}\cdot\ldots\cdot b_1^{-1}$;
    \item $h=b_1\cdot a_2\cdot\ldots\cdot a_k\cdot b_k$, we obtain $a\cdot b_1\cdot\ldots\cdot a_k\cdot (b_k b b_k^{-1}) \cdot a_k^{-1}\cdot\ldots\cdot b_1^{-1}$;
    \item $h=b_1\cdot a_2\cdot\ldots\cdot a_k$, we obtain $a\cdot b_1\cdot\ldots\cdot a_k\cdot b \cdot a_k^{-1}\cdot\ldots\cdot b_1^{-1}$;
  \end{enumerate}

  Next, suppose $g$ is conjugate to a non-peripheral element $a'\in A$, say $g=ha'h^{-1}$, and $ag$ is conjugate to an element in a factor. We again consider a reduced word representing $h$. If this is a single letter in $A$, i.e., $h=a_1\in A$, then $g=a_1 a' a_1^{-1}\in A$ as claimed. In all other cases, we see above that $ag$ is represented by a cyclically reduced word that is not in a factor, a contradiction. Again, we list the (four) cases and the cyclically reduced words we obtain:
  \begin{enumerate}
    \item $h=a_1\cdot b_1\cdot\ldots\cdot a_k\cdot b_k$, $k\geq 1$, we obtain $(a_1^{-1}aa_1)\cdot b_1\cdot\ldots\cdot a_k\cdot b_k \cdot a' \cdot b_k^{-1}\cdot\ldots\cdot b_1^{-1}$;
    \item $h=b_1\cdot a_2\cdot\ldots\cdot a_k\cdot b_k$, we obtain $a\cdot b_1\cdot\ldots\cdot a_k\cdot b_k \cdot a' \cdot b_k^{-1}\cdot\ldots\cdot b_1^{-1}$;
    \item $h=b_1\cdot a_2\cdot\ldots\cdot b_{k-1}\cdot a_k$, we obtain $a\cdot b_1\cdot\ldots\cdot a_k\cdot b_{k-1} \cdot (a_k a' a_k^{-1}) \cdot b_{k-1}^{-1}\cdot\ldots\cdot b_1^{-1}$;
    \item $h=a_1\cdot b_1\cdot a_2\cdot\ldots\cdot b_{k-1}\cdot a_k$,  we obtain $(a_1^{-1}aa_1)\cdot b_1\cdot\ldots\cdot (a_k a' a_k^{-1})\cdot \ldots\cdot b_1^{-1}$ which is not in a factor if $k> 1$, i.e., except for the case $k=1$, $h=a_1$ (mentioned above) where $g=a_1 a' a_1^{-1}\in A$.
  \end{enumerate}

\end{proof}

We will apply the above with $C$, the fundamental group of a simple closed curve $\delta$ in a surface $F$ that separates $F$, with $A$ and $B$ the fundamental groups of the closures of the components of $F\setminus \delta$. We use Lemma~\ref{split} in this context.

Recall that $V'$ is a component of $\overline {K_m'\setminus K_{m-1}'}$ with $\delta'=V'\cap K_{m-1}'$ and $V$ is a component of $\overline{\Sigma\setminus K_{m-1}}$ such that $f_*(\delta')$ is freely homotopic to a curve in $V$.
Further, $q'$ is a point in $\delta'$ and $q=f(q')$. The following lemma says that $f$ sends each loop in $V'$ based at $q'$ to a loop that is homotopic fixing basepoint to a loop in $V$ based at $q$, up to base-point fixing homotopy.

\begin{lemma}\label{pi_component_map}
  We have $ f_*(i'_*(\pi_1(V', q')))\subset i_*(\pi_1(V, q)).$
\end{lemma} 

\begin{proof}
  Let $\gamma'$ be a curve in $V'$ based at $q'$, so $[\gamma']\in i'_*(\pi_1(V', q'))$. We show that $f_*([\gamma'])\in i_*(\pi_1(V, q))$. If $\gamma'$ is \emph{peripheral}, i.e., homotopic to a power of the boundary component $\delta'$, then either $\gamma'=\delta'^k$ for some $k\in\Z$ or $\gamma'=g'*\delta'^k*g'^{-1}$ for some $g'\in \pi_1(V', q')$ which is not peripheral. In the first case, $f_*([\gamma'])\in i_*(\pi_1(V, q))$ as $f_*([\delta'])=(g_{m-1})_*([\delta']) \in i_*(\pi_1(V, q))$. In the second case, it suffices to show that the non-peripheral curve $g$ is mapped to a curve in $i_*(\pi_1(V, q))$.
  Hence, we can assume that $\gamma'$ is not peripheral.

  By Lemma~\ref{loop_outside}, there is a curve $\gamma\subset \Sigma\setminus K_{m-1}$ so that $f_*(\gamma')$ is homotopic to $\gamma$. Thus, $\gamma\subset \widehat{V}$ for some component $\widehat{V}$ of $\overline{\Sigma\setminus K_{m-1}}$. Let $\widehat{\delta} = \widehat{V}\cap K_{m-1}$, and let $\widehat{q}\in\widehat{\delta}$. We modify $\gamma$ by a homotopy so that $\gamma$ intersects $\widehat{\delta}$ in the single point $\widehat{q}$.
 
  We claim that $\gamma$ is not peripheral, i.e., $\gamma$ is not homotopic to a curve with image in $\partial\widehat V=\widehat \delta$. Namely, as $\gamma'$ is a curve in $V'$ that is not peripheral, there is a simple closed curve $\beta'$ in $V'$ so that $\gamma'$ and $\beta'$ are not homotopic to disjoint curves. Hence, by statement~(\ref{it:gold-scc}) of Theorem~\ref{goldman}, we have $[\beta', \gamma']\neq 0$. The curve $f_*(\beta')$ is homotopic to a curve $\beta$ in $\Sigma\setminus K_{m-1}$. If $\gamma$ were peripheral, then $\beta$ and $\gamma$ would be homotopic to disjoint curves, hence $[\beta, \gamma]$ would become $0$. But $[\beta', \gamma']\neq 0$ implies $[\beta, \gamma]\neq 0$ as $f$ commutes with the Goldman bracket, this is a contradiction. Hence, $\gamma$ is not peripheral. 

  We next see that $\widehat{V}=V$. Suppose not, then $\gamma\subset \widehat{V}$ for a component $\widehat{V}\neq V$ of $\overline{\Sigma\setminus K_{m-1}}$.
  As $\gamma$ is homotopic to $f_*(\gamma')$, there exists a path $\theta$ in $\Sigma$ from $q$ to $\widehat{q}$ such that $f_*(\gamma')$ is homotopic in $\Sigma$ fixing base-point to $\theta *\gamma * \bar{\theta}$. Observe that $\pi_1(\Sigma, \widehat{q})$ is the amalgamated free product
 \begin{equation}
     \pi_1(\Sigma, \widehat{q}) = \pi_1(\widehat{V}, \widehat{q}) *_{\pi_1(\widehat{\delta}, \widehat{q})} \pi_1(\overline{\Sigma\setminus \widehat{V}}, \widehat{q}). \label{FirstAmalgamated}
 \end{equation}

  We identify $\pi_1(\Sigma, q)$ with $\pi_1(\Sigma, \widehat{q})$ using the change of base-point isomorphism determined by $\theta$, i.e., the isomorphism $[\lambda] \mapsto [\bar{\theta}*\lambda*\theta]$ for a loop $\lambda$ based at $q$. Under this identification, $f_*([\gamma']) \equiv [\gamma]$.
  Further, as $f_*(\delta')$  is homotopic to a curve $\eta$ in $V\subset \Sigma\setminus\widehat{V}$, $f_*([\delta'])$ is conjugate to an element  of $\pi_1(\overline{\Sigma\setminus\widehat{V}}, \widehat q)$, which is non-peripheral in $\Sigma\setminus\widehat{V}$ as $K_{m-1}$ is not an annulus or disc by assumption.
  Hence, with respect to the amalgamated free product given by Equation~\ref{FirstAmalgamated},  $f_*([\gamma' * \delta'])=f_*([\gamma'])\cdot f_*([\delta'])$ is of the form $ag$ of Lemma~\ref{amalgamated_free_product}, statement~(\ref{it:diff}).  It follows that $f_*(\gamma'*\delta')$ is not homotopic to a curve that is disjoint from $\widehat{\delta}$ using Lemma~\ref{split} (see Figure~\ref{drawing}). This contradicts Lemma~\ref{loop_outside} as  $\widehat{\delta}\subset \partial K_{m-1}$ and $f_*(\gamma'*\delta')$ is homotopic to a curve in $\Sigma\setminus K_{m-1}$. Thus, $\widehat{V}=V$.

\begin{figure}$$\adjustbox{trim={0.0\width} {0.0\height} {0.0\width} {0.0\height},clip}{\def\svgwidth{\linewidth}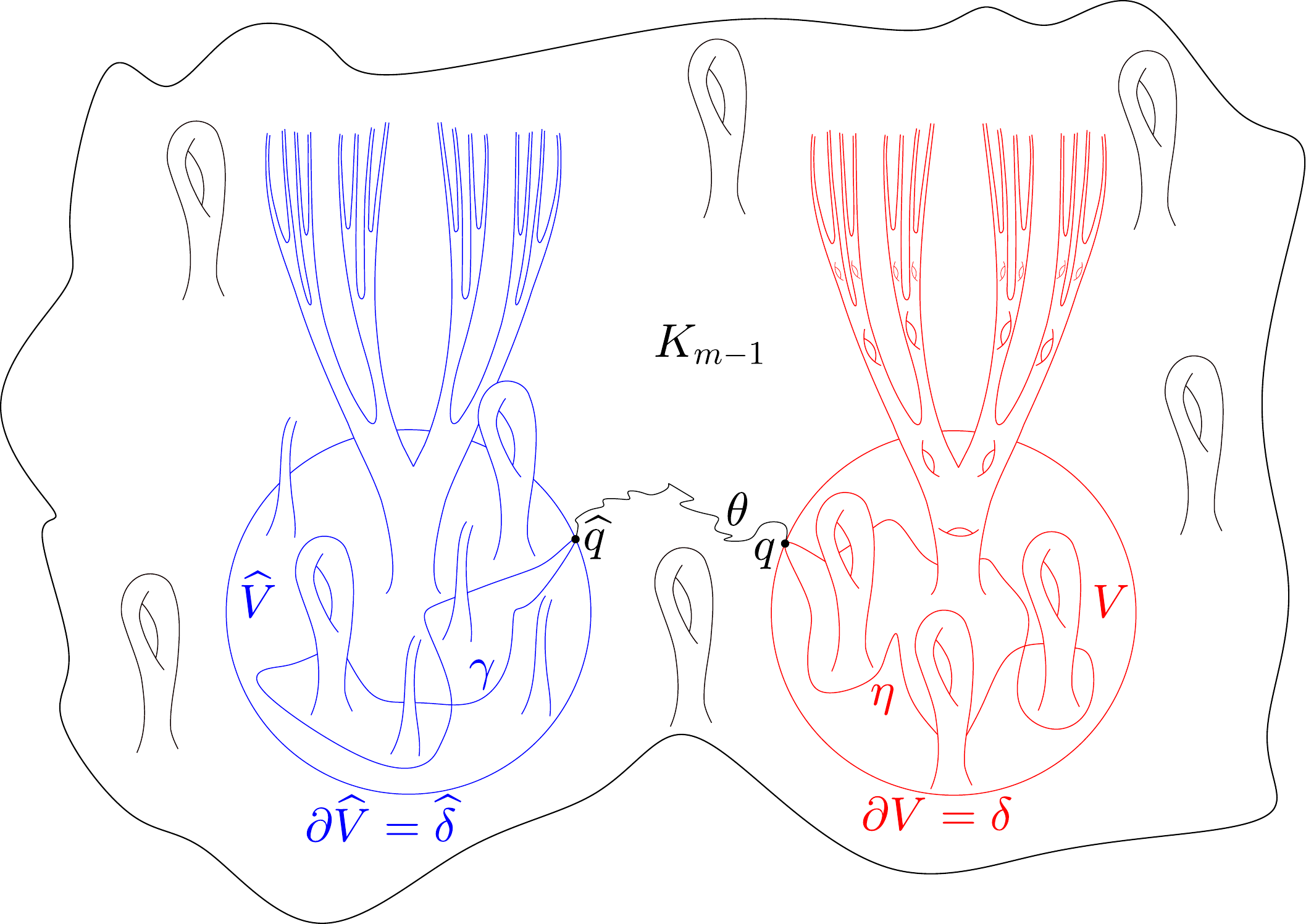}$$\caption{If $\widehat V\neq V$, $f_*([\gamma'])\cdot f_*([\delta])=[\gamma*\theta*\eta*\bar\theta]$ is not homotopic to a curve disjoint from $\del\widehat{V}$, giving a contradiction.}\label{drawing}
\end{figure}

Next, we can express the fundamental group $\pi_1(\Sigma, q)$ as an amalgamated free product
\begin{equation}
    \pi_1(\Sigma, q) = \pi_1(V, q) *_{\pi_1(\delta, q)} \pi_1(\overline{\Sigma\setminus V}, q). \label{SecondAmalgamated}
\end{equation}

  Next, as $\gamma\subset V$, modifying by a homotopy, we can assume that $\gamma$ is a loop based at $q$. As $\gamma$ is freely homotopic to $f_*(\gamma')$, $[\gamma]\in \pi_1(\Sigma, q)$ is conjugate to the element $f_*([\gamma'])\in \pi_1(\Sigma, q)$, which is not peripheral as $\gamma$ is not peripheral. We claim that $f_*([\delta'])$ is also not peripheral. Namely, as $V$ contains a curve that is not peripheral, namely $\gamma$, the boundary component $\delta$ of $V$ which is contained in $K_m$ does not bound a cylinder or a disc. Hence, one of the curves $\alpha_i$ in Lemma~\ref{subsurface} cannot be homotoped to be disjoint from $\delta$. But $\delta'$ is disjoint from $\alpha_i'$, so $[\delta', \alpha_i']\neq 0$, but $[f_*(\delta'), \alpha_i] =0$. As in the case of $\gamma'$, we conclude that $f_*([\delta'])$ is not peripheral (as it is not conjugate to a power of $\delta$).

 For the amalgamated free product decomposition~\ref{SecondAmalgamated}, we apply statement~(\ref{it:same}) of Lemma~\ref{amalgamated_free_product} with $a=f_*([\delta'])$ and $g=f_*([\gamma'])$. As $f_*(\delta' * \gamma')$ is conjugate to an element disjoint from $K_{m-1}$, hence from $\delta$, $f_*([\delta' *\gamma'])$ is conjugate to an element in $\pi_1(V, q)$ or  $\pi_1(\overline{\Sigma\setminus V}, q)$, i.e., an element of a factor. It follows by statement~(\ref{it:same}) of Lemma~\ref{amalgamated_free_product} that $f_*([\gamma'])\in \pi_1(V, q)$, as desired.
 \end{proof}

\subsection{Maps on complementary components}\label{S:complementary}
The rest of the construction on each component is analogous to the base case with some refinements.

We continue with the notation of Section~\ref{pi_component_map}.
On the component $V'=V'^{(j)}$ of $\overline{K_m\setminus K_{m-1}}$, we construct a map $g_m^{(j)}\co V'\to V$ extending $g_{m-1}\vert_{\delta'}$ so that $i_*\circ (g_m^{(j)})_*=(f\vert_{V'})_*\co\pi_1(V', q')\to \pi_1(\Sigma, q)$ and $g_m^{(j)}(\partial K_m'\cap V')\subset V\subset \Sigma\setminus K_m$.

Namely, we extend the graph in $V'=V'^{(j)}$ consisting of a single vertex $q'$ and the edge $\delta'$ to a graph $\Gamma'$ so that $V'$ deformation retracts to $\Gamma'$ (this exists as $V'$ has more than one boundary component as a consequence of Theorem~\ref{subsurface}, statement~(\ref{it:noncpt})). By Lemma~\ref{pi_component_map}, we can extend $g_{m-1}\vert_{\delta'}$ to a map $\varphi\co \Gamma'\to V$ so that $\varphi_*=(f\vert_{\Gamma'})_*\co\pi_1(\Gamma', q')\to \pi_1(\Sigma, q)$ (as any homomorphism from the fundamental group of a graph to that of a topological space is induced by a map). A regular neighbourhood $\mathcal{N}(\Gamma')$ of $\Gamma'$ has a deformation retraction $\rho\co \mathcal{N}(\Gamma')\to \Gamma'$. We get a map $g_m^{(j)}\co \mathcal{N}(\Gamma')\to \Sigma$ by setting $g_m^{(j)}\vert_{\mathcal{N}(\Gamma')}=\varphi\circ \rho$ (see Figure~\ref{complicated}).

Next, $\overline{V'\setminus \mathcal{N}(\Gamma')}$ is a disjoint union of annuli $C_i'\times [0, 1]$, $1\leq i\leq k$, so that the boundary components of $K_m'$ other than $\delta'$ are the curves $C_i' \times \{1\}$, $1\leq i\leq k$ and the boundary components of $\mathcal{N}(\Gamma')$ are $C_i' \times \{0\}$, $1\leq i\leq k$. By Lemma~\ref{subsurface}, $f(C_i'\times \{1\})$ is homotopic to a curve $\gamma\subset \Sigma\setminus K_m$ for each $i$, $1\leq i\leq k$. We define $g_m^{(j)}$ on $C_i'\times \{1\}$ to map to $\gamma$.
As $f(C_i'\times \{0\})$ is homotopic to $f(C_i'\times \{1\})$ and $g_m^{(j)}(C_i'\times \{0\})$, we deduce that $\gamma$ is homotopic to $g_m^{(j)}(C'\times \{0\})$ in $\Sigma$, and hence in $\Sigma\setminus K_{m-1}$by Lemma~\ref{split}. We use this homotopy to extend $g_m^{(j)}$ to $C_i'\times [0, 1]$.

We also extend the tree $T_{m-1}'$ to a tree $T'^{(j)}_m$ and modify $f$ as in the base case. Let the boundary components of $V'$ other than $\delta'$ be $\delta_1', \dots, \delta_n'$, $\delta_i' = C_i'\times \{1\}$. Pick a point $q_i'\in \delta_i'$ for each $i$, $1\leq i\leq n$. Pick a family of arcs $\theta_i'$ with interiors disjoint so that $\theta_i'$ is an arc from $q'$ to $q_i'$ in $K_1'$. Let $T_m'^{(j)}$ be the union of these arcs. As $T_m'^{(j)}$ is contractible we can homotope $f$ fixing $\Sigma'\setminus (\textrm{int}(V'))$ so that $f\vert T_m'^{(j)}=g_m^{(j)}\vert T_m'^{(j)}.$

\begin{figure}$$\adjustbox{trim={0.0\width} {0.0\height} {0.0\width} {0.0\height},clip}{\def\svgwidth{\linewidth}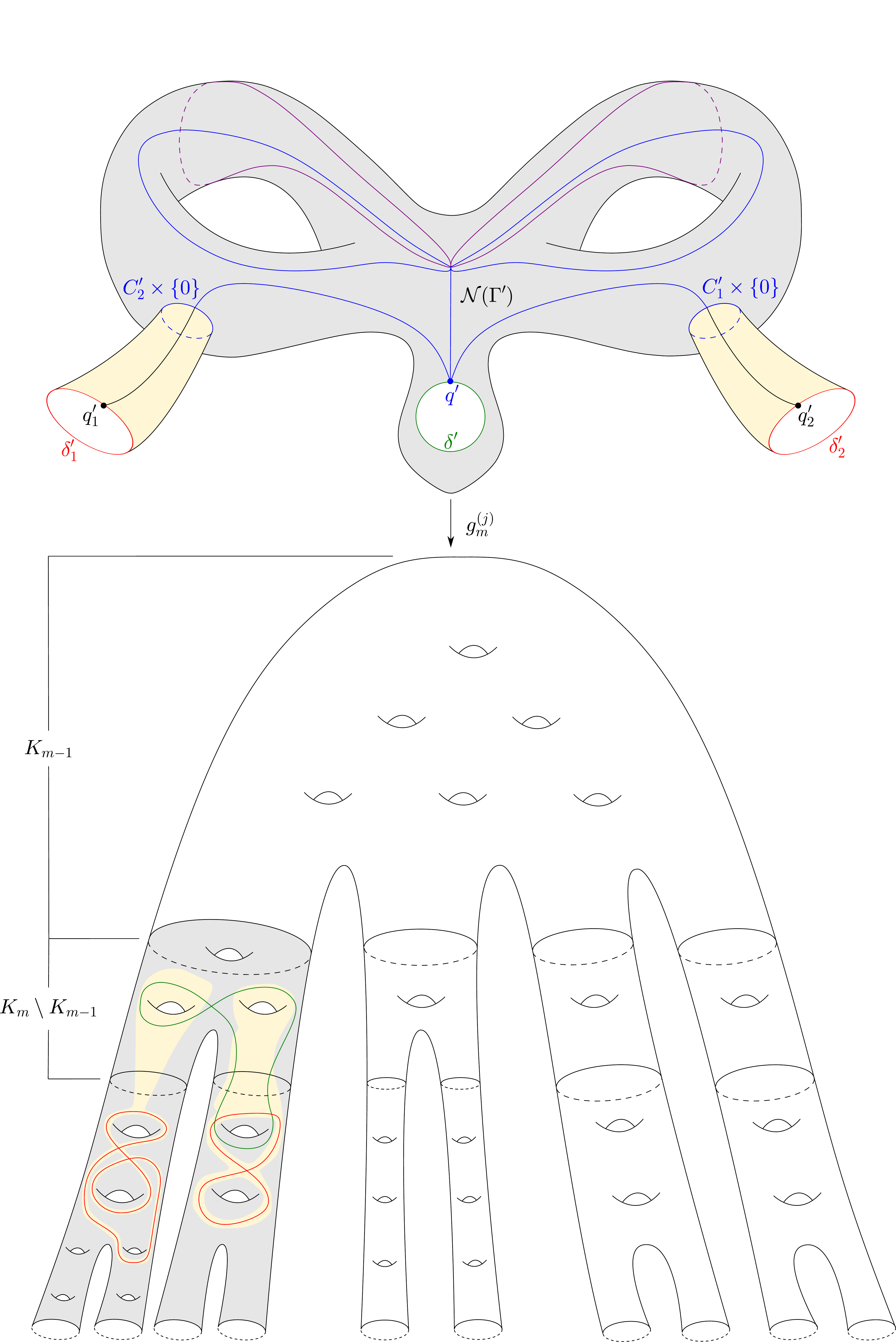}$$\caption{Description of $g^{(j)}_m\co V'=V^{(j)}_m\to V\subset\Sigma\setminus K_{m-1}$. On the top, i.e., in $V'$, the purple, blue, and green portions (arcs and circles) form $\Gamma'$, and blue and black arcs form $T'^{(j)}_m$. At the bottom: the grey and yellow shades indicate $g^{(j)}_m(V')$ and $g^{(j)}_m(C_i'\times [0,1])$, respectively, and red and green loops denote $g^{(j)}_m(C_i'\times \{1\})$ and $g^{(j)}_m(\delta')$, respectively, where $i=1,2$.}\label{complicated}
\end{figure}

Note that for $1\leq i\leq n$,  $f(q_i')=g_m^{(j)}(q_i')$. Further, for $1\leq i\leq n$, as $\left(f\vert _{V'}\right)_*=\left(g_m^{(j)}\right)_*\co \pi_1\left(V', q'\right) \to \pi_1(\Sigma, f(q'))$, and $f\vert_{\theta_i'}=g_m^{(j)}\vert_{\theta_i'}$, we deduce by using the change of base-point isomorphisms corresponding to $\theta_i'$ and $f\circ {\theta_i'}=g_m^{(j)}\circ {\theta_i'}$ that $\left(f \vert_{V'}\right)_*=(g_m^{(j)})_*\co \pi_1\left(V', q_i'\right) \to \pi_1(\Sigma, f(q_i'))$.

\subsection{Combining the maps}\label{S:combining}

We define $g_m \co K_m'\to \Sigma$ as the unique map extending $g_{m-1}$ whose restriction to the component $V'^{(j)}$ of $\overline{K_m'\setminus K_{m-1}'}$ is $g_m^{(j)}$. Observe that $g_m(K_m'\setminus K_{m-1}')\subset \overline{\Sigma\setminus K_{m-1}}$ and $g_m(\del K_m') \subset \Sigma\setminus K_m$.

The tree $T_m'$ is $T_{m-1}'\cup \bigcup_j\left(T_m'^{(j)}\right)$. As the homotopies of $f$ had disjoint support, we take $f$ to be the result of the composition of the homotopies. We have $f\vert_{T_m'}=g_m\vert_{T_m'}$ and the required equalities on the fundamental groups.

To summarize, we have constructed a subsurface, a map, and a tree satisfying the following properties:
\begin{lemma}\label{construction}
  The subsurface $K'_m$, map $g_m$ and tree $T'_m$ satisfy the following properties:
  \begin{enumerate}
    \item The subsurface $K_{m}'$ satisfies the conditions of Lemma~\ref{subsurface}.
    \item $g_{m}(\del K_{m}')\subset \Sigma\setminus K_{m}$.
    \item $f\vert_{T'_{m-1}}=g_{m-1}\vert_{T'_{m-1}}$.
    \item The terminal vertices of $T'_{m-1}$ are in $\del K_{m-1}'$, with one terminal vertex in each component.
    \item For a terminal vertex $q'\in T'_{m-1}$, $(f\vert_{K_{m-1}'})_*=(g_{m-1})_*:\pi_1(K'_{m-1}, q') \to \pi_1(\Sigma, q)$.
    \item\label{it:proper} $g_m(K_m'\setminus K_{m-1}')\subset \overline{\Sigma\setminus K_{m-1}}$.
  \end{enumerate}
\end{lemma}

The conditions except~\ref{it:proper} are exactly those of Section~\ref{S:inductive}, i.e., the conditions for the inductive step. 

\subsection{Constructing the proper map}

As $K_m'\supset L_m'$ by Theorem~\ref{subsurface} and the subsurfaces $L_m'$ form an exhaustion, so do the subsurfaces $K_m'$. We define the map $g\co \Sigma'\to \Sigma$ as the direct limit of the maps $g_m\co K_m'\to \Sigma$. 

We claim that $g$ is proper. It suffices to show that $g^{-1}(K_n)$ is compact for  all $n$. We claim that $g^{-1}(K_n)\subset K_{n+1}'$, and hence is a closed subset of a compact space, so compact. Namely, suppose $x\in g^{-1}(K_n)$ and $x\notin K_{n+1}'$. Then $x\in K_{m}'\setminus K_{m-1}'$ for some $m > n + 1$. Hence, $g(x)= g_m(x)\subset \overline{\Sigma\setminus K_{m-1}}$ by statement~(\ref{it:proper}) of Lemma~\ref{construction}, so $g(x)\notin \textrm{int}(K_{m-1})\supset K_{m-2}\supset K_n$ (the last containment by $m>n+1$), a contradiction.

Further, $g_\ast=f_\ast\co\pi_1(\Sigma', p')\to \pi_1(\Sigma, p)$, so $g$ is homotopic to $f$. By~\cite[Theorem 1]{arxiv1}, $g$ is in turn (properly) homotopic to a homeomorphism, and hence so is $f$.

As $g$ commutes with the Goldman bracket, we deduce that $g$ is orientation preserving. Namely, first, modify $g$ by an isotopy to ensure that $g$ is a diffeomorphism (in dimension $2$, homeomorphisms are isotopic, hence properly homotopic to diffeomorphisms). If $g$ is orientation reversing, then as the signs of intersection numbers are reversed, we get the equation $\left[g_*(x'),g_*(y')\right]=-g_*\left([x',y']\right)$. As $f_*=g_*$ and $f_*$ commutes with the Goldman bracket, we conclude that $\left[g_*(x'),g_*(y')\right]=0$ for all $x$, $y$. As $g_*$ is an isomorphism, it follows that the Goldman bracket is trivial on $\Sigma$, a contradiction as $\Sigma$ is assumed not to be a plane or cylinder, so there exist classes $x$ and $y$ with $[x, y]\neq 0$.

This completes the proof of Theorem~\ref{main}.

\section{Proof of Theorem~\ref{intersection}}

We sketch the modifications needed to prove Theorem~\ref{intersection}.  A homeomorphism preserves intersection numbers, so clearly (\ref{it:homeo}) implies (\ref{it:inter}) and (\ref{it:zero-inter}). Further, clearly (\ref{it:inter}) implies (\ref{it:zero-inter}). We show that (\ref{it:zero-inter}) implies (\ref{it:homeo}).

Assume now that, $I_\Sigma\left(f_*(x'), f_*(y')\right)=0 \iff I_{\Sigma'}(x', y')=0$ for all $x', y'\in \widehat\pi(\Sigma')$. We see that the proof of Theorem~\ref{main} goes through with this in place of the hypothesis that $f$ preserves the Goldman bracket. 

Namely, there are only two places where the hypothesis that the Goldman bracket is preserved are used directly:
\begin{enumerate}
  \item In Lemma~\ref{loop_outside} to show that if the closed curve $\gamma'$ can be homotoped to be disjoint from $\alpha'_i$, i.e., $I_{\Sigma'}(\gamma', \alpha'_i)=0$, then $I_\Sigma(\gamma, \alpha_i)=0$ where $\gamma$ is homotopic to $f_*(\gamma')$ and $\alpha_i$ is homotopic to $f_*(\alpha_i')$. Here an additional hypothesis is that $\alpha_i$ is a simple closed curve.
  \item In Lemma~\ref{pi_component_map}, to show that if the closed curve $\gamma'$ can be homotoped to be disjoint from $\beta'$, i.e., $I_{\Sigma'}(\gamma', \beta')=0$, then $I_\Sigma(\gamma, \beta)=0$ where $\gamma$ is homotopic to $f_*(\gamma')$ and $\beta$ is homotopic to $f_*(\beta)$. Here the additional hypothesis is that $\beta'$ is a simple closed curve.
\end{enumerate}

Clearly both these follow under the hypothesis that $I_\Sigma\left(f_*(x'), f_*(y')\right)=0 \iff I_{\Sigma'}(x', y')=0$ for all $x', y'\in \widehat\pi(\Sigma')$. Thus, the proof of Theorem~\ref{main} goes through with this in place of the hypothesis that $f$ preserves the Goldman bracket.
\subsection*{Acknowledgment}
The first and the last-named authors are supported by fellowships from the National Board for Higher Mathematics. We are grateful to the anonymous referee for his careful reading of the paper and his comments and suggestions, which helped considerably in improving the manuscript.
\bibliographystyle{plain}
\bibliography{references.bib}

\end{document}